\documentclass[11pt]{article}
\usepackage{amsthm, amsmath, amssymb, amsfonts, url, booktabs, tikz, setspace, fancyhdr, bm}
\usepackage{hyperref}
\usepackage{geometry}
\usepackage{hyperref, enumerate}
\usepackage[shortlabels]{enumitem}
\usepackage[babel]{microtype}
\usepackage[english]{babel}
\usepackage[capitalise]{cleveref}
\usepackage{comment}
\usepackage{bbm,tkz-graph,subcaption}
\usepackage{csquotes}
\usepackage{mathrsfs}
\usepackage{mathabx}
\usetikzlibrary{patterns}
\usetikzlibrary{shapes}
\usepackage{bbm}

\usepackage[utf8]{inputenc} 
\usepackage[T1]{fontenc}
\usepackage{amsfonts}
\usepackage{amscd}
\usepackage{graphicx}
\usepackage{enumitem}
\usepackage{verbatim}
\usepackage{hyperref}
\usepackage{amsmath,caption}
\usepackage{url,pdfpages,xcolor,framed,color}
\usepackage{todonotes}
\usepackage{comment}

\newtheorem{thr}{Theorem}[section]
\newtheorem{q}[thr]{Question}

\newtheorem{lem}[thr]{Lemma}
\newtheorem{prop}[thr]{Proposition}
\newtheorem{conj}[thr]{Conjecture}
\theoremstyle{definition}
\newtheorem{defi}[thr]{Definition}
\newtheorem*{defi*}{Definition}

\newtheorem{claim}[thr]{Claim}

\def\Z{\mathbb{Z}}

\def\BG{S}

\newcommand*{\floorfrac}[2]{\mathopen{}\left\lfloor\frac{#1}{#2}\right\rfloor\mathclose{}}

\newcommand*{\myproofname}{Proof}
\newenvironment{claimproof}[1][\myproofname]{\begin{proof}[#1]}{\end{proof}}

\title{Note on extremal problems about connected subgraph sums}

\date{}
\author{Stijn Cambie\thanks{Department of Computer Science, KU Leuven Campus Kulak-Kortrijk, 8500 Kortrijk, Belgium. Supported by a postdoctoral fellowship by the Research Foundation Flanders (FWO) with grant number 1225224N. Email: \protect\href{mailto:stijn.cambie@hotmail.com}{\protect\nolinkurl{stijn.cambie@hotmail.com}}}
\and
Carla Groenland\thanks{Delft Institute of Applied Mathematics, TU Delft, the Netherlands. Supported by the Dutch Research Council (NWO, VI.Veni.232.073). Email: \protect\href{mailto:c.e.groenland@tudelft.nl}{\protect\nolinkurl{c.e.groenland@tudelft.nl}}}
}

\begin{document}
\maketitle
\begin{abstract}
    For a graph $G$ with vertex assignment $c:V(G)\to \mathbb{Z}^+$, we define $\sum_{v\in V(H)}c(v)$ for $H$ a connected subgraph of $G$ as a connected subgraph sum of $G$.
    We study the set $S(G,c)$ of connected subgraph sums and, in particular, resolve a problem posed by Solomon Lo in a strong form. We show that for each $n$-vertex graph, there is a vertex assignment $c:V(G)\to \{1,\dots,12n^2\}$ such that for every $n$-vertex graph $G'\not\cong G$ and vertex assignment $c'$ for $G'$, the corresponding collections of connected subgraph sums are different (i.e., $S(G,c)\neq S(G',c')$). 
    We also provide some remarks on vertex assignments of a graph $G$ for which all connected subgraph sums are different.
\end{abstract}

\section{Introduction}
For a graph $G$, a subgraph $H \subset G$ and a vertex-assignment $c \colon V(G) \to \mathbb Z^+$ with $\Z^+=\{1,2,3,\dots\}$, we denote $c(H)=\sum_{v \in V(H)} c(v)$ as the \textit{subgraph sum} of $H.$
The set 
\[
S(G,c)= \{c(H) \mid H  \text{ connected subgraph of }G\}
\]
is the set of all subgraph sums over all
connected (induced) subgraphs of $G$ for vertex-assignment $c.$ 

For example, when $G=K_n$ has vertex set $\{1,\dots,n\}$, then 
\[
S(K_n,c)= \left\{\sum_{i\in U}c(i): U\subseteq \{1,\dots,n\}\right\},
\]
is the collection of subset sums of the positive integers $c(1),\dots,c(n)$.

For a tree $T$, $S(T,c)$ is also called the tree spectrum~\cite{Lo2023}. The study of the tree spectrum in~\cite{Lo2023} was inspired by a connection to the cycle spectrum of planar Hamiltonian graphs.
Solomon Lo~\cite{Lo2023} studied the number of values in a certain range (interval) belonging to $S(T,c)$.
Relatedly, one can remark that no more than $\binom{n}{\floorfrac n2}$ many values can appear in $S(G,c)$ for interval lengths smaller than $\min_{v \in V(G)} c(v)$, by the Littlewood-Offord theorem or Sperner's theorem.
The latter is still true when considering the set $S(G,c)$ as a multiset.

It is natural to expect that $G$ has some influence on which $S(G,c)$ are possible (for different choices of $c$) and that this may depend on the size of the ``vertex weights'' that are allowed. For a graph $G$ and positive integer $N$, let $\BG(G;N)=\{S(G,c)\mid c:V(G)\to \mathbb{Z}^+ \text{ with }c(G)=N\}.$
Solomon Lo~\cite{Lo2023} (at the end of the conclusion section) made the following conjecture. 
\begin{conj}
For any two non-isomorphic connected graphs $G_1$ and $G_2$ on $n$ vertices, $\BG(G_1; 2^n-1) \not=  \BG(G_2; 2^n-1).$
\end{conj}

We prove this conjecture and strengthen it in two directions. 

Firstly, we show that it is even possible to select an element in $\BG(G;2^n-1)$ that is unique to $G$.
\begin{thr}
\label{thm:2^n-1}
    Let $G$ be a connected graph on $n$ vertices. Then there exists $c\colon V(G)\to \{1,\dots,2^n-1\}$ with $c(G)=2^n-1$ such that $S(G,c)=S(G',c')$ for an $n$-vertex graph $G'$ implies that $G\cong G'$. 
\end{thr}
The result above immediately implies the conjecture. For large $n$, we also reduce $2^n-1$ to $12n^3$.
\begin{thr}\label{thr:12n^3}
 Let $G$ be a graph on $n$ vertices. Then there exists $c\colon V(G)\to \{1,\dots,12n^2\}$ such that $S(G,c)=S(G',c')$ for an $n$-vertex graph $G'$ and $c':V(G)\to \mathbb{Z}^+$ implies that $G\cong G'$. When $G$ is connected, the assumption that $G'$ has $n$ vertices may be replaced by the assumption that $G'$ is connected. 
\end{thr}
In other words, we can ``reconstruct'' $G$ from knowing that there is a vertex assignment $c\colon V(G)\to \mathbb{Z}^+$ with a particular set of connected subgraph sums $S(G,c)$, and only ``modest'' weights on the vertices are required for $c$, i.e., $c(G)=O(n^3).$ 
There is a wide literature on graph reconstruction questions, perhaps the most similar to ours being the reconstruction from the set of vertex sets of size $k$ that induce a connected subgraph~\cite{bastide2023reconstructing,kluk2024graph}. 

Note that for any graph $G$ and assignment $c$, the graph $G'=G+G$ obtained by taking the disjoint union of two copies of $G$, with $c'$ defined as $c$ on both copies, has $S(G,c)=S(G',c')$. So some assumption of connectivity or on the number of vertices is required in Theorem~\ref{thr:12n^3}.

The proofs of our main results are presented in~\cref{sec:proofs}. We make some observations about vertex assignments for which all connected subgraphs yield distinct sums in Section~\ref{sec:SSD}, which are related to the Erd\H{o}s distinct subset sum problem and optimal Golomb rulers.
In~\cref{sec:conc}, we discuss potential directions for future work.

\section{Reconstruction from collection of connected subgraph sums}\label{sec:proofs}
We show the following result which immediately implies Theorem~\ref{thm:2^n-1}.
\begin{lem}
\label{lem:2n}
Let $G=(V,E)$ be a connected graph with $V=\{v_1,\dots,v_n\}$.
% such that there are no $i<j<k$ with $v_i$ and $v_k$ in connected component $C$ of $G$ and $v_j$ in connected component $C'\neq C$. 
Let $c(v_i)=2^{i-1}$ for $i\in [n]$. Then $G\cong G'$ for every $n$-vertex graph $G'$ and $c'\colon V(G)\to \mathbb{Z}^+$ with $S(G,c)=S(G',c')$.
\end{lem}
\begin{proof}
We order the vertices $w_1,\dots,w_n$ of $G'$ such that $c'(w_1)\leq \dots \leq c'(w_n)$. 

We first show by induction on $i$ that $c'(w_i)\leq 2^{i-1}$. This is true for $i=1$ because $1=c(v_1)\in S(G,c)=S(G',c')$ and $c'$ only assigns positive integer values. Assuming the claim holds for $w_1,\dots,w_j$, this implies that $\sum_{i=1}^j c'(w_j)<2^{j}$. So the value $2^j\in S(G',c')$ must be coming from a connected subgraph with a vertex that is not among $w_1,\dots,w_j$ and therefore $c'(w_{j+1})\leq 2^j$.

Next, since $2^n-1\in S(G,c)=S(G',c')$, we find that $c'(w_{j})=2^{j-1}$ for all $j$. This  implies that $f:V(G)\to V(G')$ with $f(v_j)=w_j$ is an isomorphism since $v_iv_j\in E(G)$ if and only if $2^{i-1}+2^{j-1}\in S(G_1,c)=S(G_2,c')$ if and only if $w_iw_j\in E(G')$. 
\end{proof}
As shown in Figure~\ref{fig:lemma_disconnected}, we cannot simply remove the assumption that $G$ is connected from the statement of Lemma~\ref{lem:2n}.

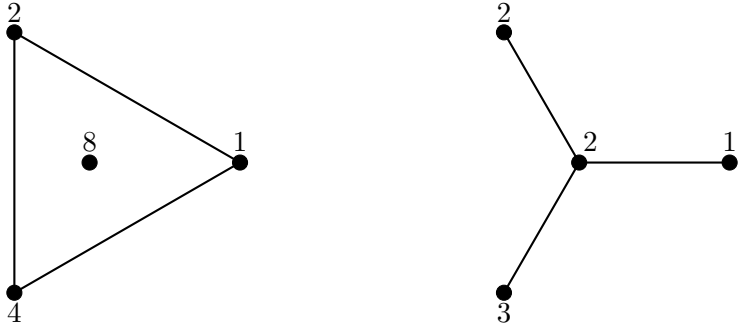
\begin{figure}[h]
    \centering
    \begin{tikzpicture}
    {
     \draw[fill] (0,0) circle (0.1);	
        \foreach \x in {0,120,240}{
            \draw[fill] (\x:2) circle (0.1);
            \draw[thick] (\x+120:2) -- (\x:2);
    }
    \draw (0,0.275) node { $8$};
 \draw (2,0.275) node { $1$};
     \draw (-1,2) node { $2$};
    \draw (-1,-2) node { $4$};
	}
	\end{tikzpicture}\quad\quad \quad \quad\quad\quad\quad\quad
 \begin{tikzpicture}
    {
    % \draw[thick] (1,1) -- (1,3)--(3,3)--(3,1)--(1,1);	
    \foreach \x in {0,120,240}{
            \draw[fill] (\x:2) circle (0.1);
            \draw[thick] (0,0) -- (\x:2);
    }
    \draw[fill] (0,0) circle (0.1);
    \draw (0.15,0.275) node { $2$};
 \draw (2,0.275) node { $1$};
     \draw (-1,2) node { $2$};
    \draw (-1,-2) node { $3$};
	}
	\end{tikzpicture}
    \caption{The graph $G$ on the left with depicted weight assignment $c$ and the graph $G'$ on the right with depicted weight assignment $c'$  satisfy $S(G,c)=\{1,2,3,4,5,6,7,8\}=S(G,c')$.}
    \label{fig:lemma_disconnected}
\end{figure}

The example $S(K_n,c)=\{0,1,\ldots,2^n-1\}=S(P_{2^n-1},c')$ where $c(v_i)=2^{i-1}$ and $c' \equiv 1$, indicates that the graphs $G$ and $G'$ having the same order is a necessary restriction as well.

Next, we prove the stronger bound of $c(G)\leq12n^3$ (where furthermore connectedness is not needed).
\begin{proof}[Proof of Theorem~\ref{thr:12n^3}]
Let $G$ be a graph on $n$ vertices. We want to show that there exists $c\colon V(G)\to \{1,\dots,12n^2\}$ such that $S(G,c)=S(G',c')$ implies that $G\cong G'$. 

The claim below follows from the Sidon set construction from Erd\H{o}s and Tur\'an~\cite{erdos1941problem}, but we include the proof for convenience to get the exact statement we use.
 %  $N \le 12n^3 .$
\begin{claim}
    There exists a set $S=\{s_1, s_2, \ldots, s_n\}\subseteq  \{1,\dots,  4n^2\}$, such that $s_i+s_{\ell}=s_j+s_k$ implies that $\{i,\ell\}=\{j,k\}$. 
\end{claim}
\begin{claimproof}
    By Bertrand's postulate, there exists an odd prime $n+1\leq p\leq 2n-1$. For an integer $x$, let $r_p(x)\in \{0,\dots,p-1\}$ denote the remainder of $x$ after division by $p$. Set
\[
s_i=2pi + r_p(i^2)
\]
for every $i \in [n]$. Now  $s_i+s_j=s_k+s_\ell$ implies
\[
2(i-j)p+ r_p(i^2)-r_p(j^2) = s_i-s_j=s_k-s_{\ell}=2(k-\ell)p+ r_p(k^2)-r_p(\ell^2). 
\]
This implies that $i-j=k-\ell$ and $r_p(i^2)-r_p(j^2)=r_p(k^2)-r_p(\ell^2)$. We are done if $i-j=0$, so suppose $i\neq j$. Since $p\geq n+1$, it follows that $i-j\not\equiv 0 \mod p$. Since $(i+j)(i-j)=i^2-j^2\equiv k^2-\ell^2\bmod p$, this shows that $i+j\equiv k+\ell\bmod p$ which using $p\geq n+1$ again shows that $i+j=k+\ell$. This now implies that $i=k$ and $j=\ell$. 
Note also that for all $i$, $s_i\leq 2pi+(p-1)\leq 2(2n-1)n+2n\leq 4n^2$.
\end{claimproof} 
Let $V(G)=\{v_1, \ldots, v_n\}$. Let $M= 4n^2$. Let $1\leq s_1\leq \dots \leq s_n\leq M$ be the values from the claim above and set $c(v_i)=2M+s_i$ for every $i \in [n]$. 
Then $c(G)\le 3nM \le 12n^3$ and $c(v_1)\leq \dots \leq c(v_n)$.

Let $G'$ and $c'$ be given such that $S(G,c)=S(G',c')$. We will show that $G\cong G'$.
Let $w_1,\dots,w_{n'}\in V(G')$ with $c'(w_1)\leq \dots\leq c'(w_{n'})$. Since $2c(v_1)\geq 4M>3M\geq c(v_n)$, it follows that the $n$ smallest weights in $S(G,c)=S(G',c')$ must be $c(v_1),\dots,c(v_n)$ and $c'(w_1),\dots,c'(w_{n})$. So $c(v_i)=c'(w_i)$ for all $i\in [n]$. If $G$ and $G'$ are both connected, then $c(v_1)+\dots+c(v_n)$ and $c'(w_1)+\dots+c'(w_{n'})$ both equal the largest value in $S(G,c)=S(G',c')$ and therefore $n=n'$ now follows. 

Next, we use that $3c(v_1)\geq 3(2M+1)>2(2M+M)\ge 2c(v_n)$. This shows that all values in $\{4M,4M+1,\dots,6M\}\cap S(G,c)$ are created via connected vertex sets of size 2 (also known as edges) in both $G$ and $G'$. By construction, $c(v_i)+c(v_j)=c(v_k)+c(v_\ell)$ implies $s_i+s_j=s_k+s_\ell$ which implies $\{i,j\}=\{k,\ell\}$ (by choice of $S$). This means that $v_iv_j\in E(G)$ if and only if $4M+s_i+s_j\in S(G,c)=S(G',c')$ if and only if $w_iw_j\in E(G')$. The map which sends $v_i$ to $w_i$ for all $i\in[n]$ gives the desired isomorphism $G \cong G'$.
\end{proof}

\section{Subgraph-sum distinct assignments}\label{sec:SSD}

The set $S(G,c)$ could contain values which appeared as the subgraph sum of multiple subgraphs. 
This is not the case for what we define as an SSD assignment $c$.

\begin{defi}
    Let $G=(V,E)$ be a graph.
    A \textit{subgraph sumset-distinct} (SSD) assignment $c$ is a mapping $c \colon V \to \mathbb Z^+$ for which every two distinct connected induced subgraphs $G_1$ and $G_2$ satisfy $c(G_1)\not=c(G_2)$. 
\end{defi}

A natural question that arises is to determine optimal choices for $c$. The following two optimality criteria for $c$ seem natural.
Let $M(G)=\min_c \sum_{v \in V} c(v)$, where the minimum is taken over all possible SSD assignments $c$.
Let $m(G)=\min_c \max\{ c(v) \mid v \in V\}$, where the minimum is taken over all possible SSD assignments $c$.

Specific cases of this have already been studied in the literature. For example, estimating $m(K_n)$ corresponds with the Erd\H{o}s distinct subset sums problem~\cite{Erdos89}. We note that $M(K_n)=2^n-1$. Axenovich, Caro and Yuster~\cite{ACY23} also studied a generalisation of $m(G)$, where the vertices of a hypergraph are assigned positive integer weights and all hyperedges need to receive distinct sums; our setting is a special case of this where the hyperedges correspond to the connected subgraphs of a fixed graph. It follows from their result that $m(G)=o(k^2)$ if $G$ is an $n$-vertex connected graph with $k=n^{O(1)}$ connected subsets.

For $G=P_n$, $M(P_n)$ is related to Golomb rulers or Sidon sets.
Recall that a subset $S\subseteq \{1,\dots,N\}$ is called a Sidon set if $x+y=z+w$ for $x,y,z,w\in S$ implies $\{x,y\}=\{z,w\}$.
\begin{prop}
    $M(P_n)$ is the smallest integer $N$ such that there is a Sidon set $S\subseteq \{1,\dots,N\}$ of size $n$. 
\end{prop}
\begin{proof}
    Let $c$ be a SSD assignment of $P_n=(V,E)$, where $V=\{v_1, v_2, \ldots, v_n\}.$
If $s_k=\sum_{i=1}^kc(v_i)$, then $s_k-s_j=\sum_{i=j}^k c(v_i)$, so the condition that all subpath sums are distinct is the same as the condition that all differences $s_k-s_i$ are distinct. Conversely, a Sidon set $\{s_1,\dots,s_n\}\subseteq \{1,\dots,N\}$ gives rise to an SSD assignment $c$ of $P_n$ with $c(P_n)\leq N$, by inductively defining $c(v_1),\dots,c(v_n)$ such that $\sum_{i=1}^jc(v_i)=s_j$.  
\end{proof}
Similarly, the value $M(C_n)$ is exactly equal to the length of optimal circular Golomb rulers.
Examples of (circular) Golomb rulers are presented in~\cref{fig:P_n}.

\begin{figure}[h]
    \centering
    \begin{tikzpicture}
    {
    \draw[thick] (1,0) -- (6,0);	
    \foreach \x in {1,2,3,4,5,6}{\draw[fill] (\x,0) circle (0.1);}
    \draw (1,0.35) node { $2$};
    \draw (2,0.35) node { $1$};
    \draw (3,0.35) node { $7$};
    \draw (4,0.35) node { $6$};
    \draw (5,0.35) node { $5$};
    \draw (6,0.35) node { $4$};
	}
	\end{tikzpicture}\quad 
 \begin{tikzpicture}
    {
    \draw[thick] (1,1) -- (1,3)--(3,3)--(3,1)--(1,1);	
    \foreach \x in {1,3}{
        \foreach \y in {1,3}{
            \draw[fill] (\x,\y) circle (0.1);
        }
    }
    \draw (0.75,0.75) node { $2$};
    \draw (3.25,0.75) node { $1$};
    \draw (3.25,3.25) node { $4$};
    \draw (0.75,3.25) node { $6$};
	}
	\end{tikzpicture}
    \caption{SSD assignments attaining $m(G)$ and $M(G)$ for $G \in \{P_6,C_4\}.$}
    \label{fig:P_n}
\end{figure}
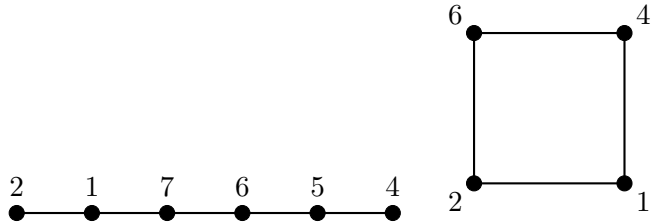

Of course, when $H$ is a subgraph of $G$, then $m(H)\leq m(G)$ and $M(H)\leq M(G)$ since the assignment used for $G$ can also be used for $H$. We also observe that, up to multiplicative constant, determining $m(G)$ for $G=S_{n+1}=K_{1,n}$ the star on $n+1$ vertices is the Erd\H{o}s distinct subset sum problem, and that $m(C_n)$ and $m(P_n)$ also have the same order.

\begin{prop}
    $m(K_n) \le m(S_{n+1}) \le 2m(K_n)$ and 
    $m(P_n) \le m(C_n) \le 2m(P_{n-1})\le 2m(P_{n})$.
\end{prop}
\begin{proof}
    To see that $m(K_n)\leq m(S_{n+1})$, note that every subset among the $n$ leaves of the star must have a unique sum (after adding the center of the star to form a connected subset). 

For $m(S_{n+1}) \le 2m(K_n)$, it is sufficient to double the values of an optimal assignment of $K_n$, assign those values to the leaves and assign an odd number to the center.

The inequalities $m(P_n) \le m(C_n)$ and $m(P_{n-1})\le m(P_{n})$ are immediate, using that if $H$ is a connected subgraph of $G$, then $m(H) \le m(G).$

To see that $m(C_n)\leq 2m(P_{n-1})$, we double the values of an optimal assignment of (the induced) $P_{n-1}$ and assign a (small) odd number to the remaining vertex of the $C_n$.
Since the subgraphs can be divided in complementary pairs, the assignment satisfies the conditions by definition.
\end{proof}
As a generalisation of the Erd\H{o}s distinct subset sums problem and Sidon sets, one can wonder about the values of $m(G)$ or $M(G)$ for a certain graph or the extremes within a specified graph class. It is unclear to the authors if there are interesting graphs $G$ for these parameters which are not $P_n$ or $K_n.$

\section{Conclusion}\label{sec:conc}
The conjecture of Solomon Lo~\cite{Lo2023} gives rise to the following question: Given $n$, what is the smallest value of $N$ such that for any two non-isomorphic graphs $G$ and $G'$ on $n$ vertices, $\BG(G;N)\neq \BG(G';N)$?
We showed that $N\leq 12n^3$ in Theorem~\ref{thr:12n^3} but it is possible that the correct order of growth is quadratic in $n$. 
In Solomon's setting, it is allowed that $\BG(G;N)\subset\BG(G';N)$. If we want to avoid this, then the cubic bound that we prove is optimal. Indeed, let $G=K_n$ and $G'=K_n^-$, a clique minus one edge. 
Let $c\colon V(G) \to \mathbb Z$ with $c(G)=N$ for some $N<\frac{n^3}{100}$. Then there are at least $0.9n$ vertices which are assigned a value below $\frac {n^2}{10}$. Since $\binom{0.9n}{2}>2\frac {n^2}{10},$ there are at least two pairs of vertices with the same sum. In particular, $G=K_n$ has two edges $e_1, e_2$ with $c(e_1)=c(e_2)$. We define $c'\in \BG(G';N)$ using $c$ and the isomorphism $G'\cong K_n\setminus e_1$. Now $S(G,c)=S(G',c')$ since the only disconnected vertex subset in $K_n \setminus e_1$ is $e_1$, but its sum is also achieved by $e_2$.

\smallskip

We introduced the stronger variant in which we wish to find a vertex assignment $c$ such that $S(G,c)$ is unique to $G$ in the sense that for any other graph $G'$ on the same number of vertices, if $S(G,c)=S(G',c')$ for some $c'$, then $G\cong G'$. By the example above and Theorem~\ref{thr:12n^3}, for $n$-vertex graphs, the smallest total vertex weight is of order $\Theta(n^3)$.
When restricting both $G$ and $G'$ to trees, we expect that a better bound may be possible.
\begin{q}
   Can~\cref{thr:12n^3} be improved when restricting to trees, in such a way that $c(G)$ is $o(n^3)$ or even $O(n^{2})$?
\end{q}

\section*{Acknowledgement}
The conjecture was shared by On-Hei Solomon Lo during the Belgian Graph Theory Conference 2023.

\paragraph{Open access statement.} For the purpose of open access,
a CC BY public copyright license is applied
to any Author Accepted Manuscript (AAM)
arising from this submission.

\end{document}